%% file: main.tex
\documentclass[11pt]{article}
\usepackage[margin=2cm]{geometry}
\usepackage{etex}
\usepackage{authblk}
\newenvironment{proof}{\paragraph{Proof:}}{\hfill$\square$}

\input{header}

\begin{document}

\title{Almost-Sure Reachability in Stochastic Multi-Mode System}

 \author[1]{Fabio Somenzi\thanks{fabio@colorado.edu}}
 \author[1]{Behrouz Touri\thanks{behrouz.touri@colorado.edu}}
 \author[2]{Ashutosh Trivedi\thanks{ashutosh.trivedi@colorado.edu}} 
 \affil[1]{Department of Electrical, Computer, and Energy Engineering, University of Colorado Boulder}
 \affil[2]{Department of Computer Science, University of Colorado Boulder}
\date{}

\maketitle

\begin{abstract}
  A constant-rate multi-mode system is a hybrid system that can switch freely
  among a finite set of modes, and whose dynamics is specified by a finite
  number of real-valued variables with mode-dependent constant rates. 
  We introduce and study a stochastic extension of a constant-rate multi-mode
  system where the dynamics is specified by mode-dependent compactly supported
  probability distributions over a set of constant rate vectors.
  Given a tolerance $\varepsilon > 0$, the \emph{almost-sure reachability}
  problem for stochastic multi-mode systems is to decide the existence of a
  control strategy that steers the system almost-surely from an arbitrary start
  state to an $\varepsilon$-neighborhood of an arbitrary target state while
  staying inside a pre-specified safety set. 
  We prove a necessary and sufficient condition to decide almost-sure
  reachability and, using this condition, we show that almost-sure reachability
  can be decided in polynomial time.
  Our algorithm can be used as a path-following algorithm in combination with
  any off-the-shelf path-planning algorithm to make a robot or an autonomous
  vehicle with noisy low-level controllers follow a given path with arbitrary
  precision. 
\end{abstract}

\section{Introduction}
\label{sec:introduction}
\input{introduction}

\section{Preliminaries}
\label{sec:prelims}
\input{prelims}

\section{Problem Formulation}
\label{sec:problem}
\input{problem}

\section{One-Dimensional Dynamics}
\label{sec:one-dim}
\input{one-dim}

\section{Higher-Dimensional Dynamics}
\label{sec:general}
\input{general}

\section{Algorithms} 
\label{sec:complexity}
\input{algos}

\section{Conclusion}
\label{sec:conclusion}
\input{conclusion}

\bibliographystyle{abbrv}
\bibliography{bib}

\end{document}

%% file: header.tex
\usepackage{amsmath}
\usepackage[usenames,dvipsnames,svgnames,table]{xcolor}
\usepackage{enumerate}
\usepackage{graphicx}
\usepackage{mathrsfs}
\usepackage{hyperref,comment}
\usepackage{subfig,xspace}

\usepackage{tikz}
\usepackage{pgf}
\usepackage{pgflibraryarrows}
\usepackage{pgffor}
\usepackage{pgflibrarysnakes}
\usetikzlibrary{fit} 
\usetikzlibrary{positioning}
\usepgflibrary{shapes}
\usetikzlibrary{snakes,automata}
\usetikzlibrary{shadows}

\tikzset{
  shadowed/.style={preaction={
      transform canvas={shift={(2pt,-1pt)}},draw opacity=.2,#1,preaction={
        transform canvas={shift={(3pt,-1.5pt)}},draw
        opacity=.1,#1,preaction={
          transform canvas={shift={(4pt,-2pt)}},draw
          opacity=.05,#1,
  }}}},
}

\makeatletter
\newif\if@restonecol
\makeatother

\usepackage[ruled,vlined,linesnumbered,lined]{algorithm2e}
\usepackage{ntheorem}

\newtheorem{theorem}{Theorem}	
\renewtheorem*{theorem*}{Theorem}	
\newtheorem{problem}{Problem}
\newtheorem{definition}{Definition}

\newtheorem{lemma}{Lemma}

\newtheorem{example}{Example}
\newtheorem{corollary}{Corollary}

\usepackage{amsmath,mathrsfs}
\usepackage{amssymb}

\newcommand{\set}[1]{\left\{ #1 \right\}}

\def\R{\mathbb{R}}

\def\Rn{\mathbb{R}^n}
\def\M{\mathcal{M}}

\def\e{\epsilon}

\def\F{\mathcal{F}}
\def\mp{\mu_{+}}
\def\mm{\mu_{-}}

\newcommand{\argmax}{\operatornamewithlimits{argmax}}

\newcommand{\smms}{\textsf{SMMS}}

\newcommand{\Tube}[6][] {
  \colorlet{InColor}{#4}
  \colorlet{OutColor}{#5}
  \foreach \I in {1,...,#3}
           {   \pgfmathsetlengthmacro{\h}{(\I-1)/#3*#2}
             \pgfmathsetlengthmacro{\r}{sqrt(pow(#2,2)-pow(\h,2))}
             \pgfmathsetmacro{\c}{(\I-0.5)/#3*100}
             \draw[InColor!\c!OutColor, line width=\r, #1]
             #6;
           }
}

\usetikzlibrary{patterns}

\tikzset{
  ashadow/.style={opacity=.25, shadow xshift=0.07, shadow yshift=-0.07},
}

\hypersetup{colorlinks=true,linkcolor={blue},citecolor={Maroon}}
           

%% file: introduction.tex
Planning and control of autonomous vehicles (or robots) are increasingly
hierarchical in nature~\cite{Firby89,Gat98} as this provides abstraction to
dissociate the complications involved in lower-level hardware control from
higher level planning decisions.
This naturally gives rise to compositional design frameworks where
the central problem is to design control so as to provide performance
guarantees for planners at higher levels by assuming performance
guarantees from the controllers at lower levels.
Le Ny and Pappas~\cite{LeNP12} recently presented a general notion of robust
motion specification at a lower level and a mechanism to sequentially compose them
to satisfy a higher-level control objective.
In this paper lower-level
controllers are abstracted as modes having constant-rate
dynamics with
stochastic noise; the control objective is to almost surely follow an arbitrary
path with an arbitrary precision.
We prove a necessary and sufficient condition ensuring the existence of such
control.

\input{motionPlanning}

In order to restrict ourselves to decidable models, we extend the constant-rate
multi-mode system framework of Alur \emph{et al}.~\cite{ATW12} by
allowing bounded stochastic uncertainties with various modes.
These systems, that we call \emph{stochastic multi-mode systems} or \smms{},
consist of a finite set of continuous variables, whose dynamics is 
given by mode-dependent constant-rates that can vary within given bounded sets
according to given probability distributions. 
This dynamics gives rise to a one-and-half player game between a controller and
the environment, where at each step the controller chooses a mode and time
duration and the environment chooses a rate vector for that mode from the given
bounded set following its distribution. 
The system evolves with that rate for the chosen time and the game continues in
this fashion from the resulting state.
A key problem for these systems is \emph{almost-sure
reachability}, which is defined as follows: given a stochastic multi-mode
system, decide whether it is possible to almost surely steer the system from any
starting state to an arbitrary neighborhood of any given target state
without it leaving a safe region.

Almost-sure reachability is a concern when solving path-planning
problem for autonomous
vehicles with finitely many modes associated with noisy dynamics.
For instance, consider the problem of navigating a robot
with a set of three motion directions
along with bounded uncertainty distributions, shown as $\mu_1, \mu_2$
and $\mu_3$ in Figure~\ref{fig:intro-example}.
An important problem for such systems is to decide so-called
\emph{$\varepsilon$-reachability property} that asks whether it is possible for
the given robot to almost surely follow a given trajectory, with arbitrary
precision, as shown by the open tube in Figure~\ref{fig:intro-example} and if
so, to compute the controller strategy. 

Our key result is that given a set of stochastic modes and a path-connected and
bounded safety set,  there is a strategy to reach an arbitrary neighborhood of
an arbitrary target state from any given state, if and only if for every
direction (vector) $\vec{v}$ there is a stochastic mode such that its expected
direction has a positive projection along the direction
$\vec{v}$. (For every mode, the expected rate is depicted as a thick
arrow in Figure~\ref{fig:intro-example}.)  
It is a straightforward consequence of this result that for
probability distributions permitting an
efficient computation of their expected values, this property can be checked in
polynomial time.
Our results can be combined with paths returned by off-the-shelf path-planning
algorithms, such as rapidly exploring random trees~\cite{rrt-plan} or Canny's
algorithm~\cite{Can88}, to accomplish motion planning in the presence of
stochastic uncertainties. 


For a detailed survey of well-known motion planning algorithms we refer the
reader to excellent expositions by Latombe~\cite{latombe2012robot}, LaValle~\cite{Lav06} and by de
Berg \emph{et al}.~\cite{BCKO08}. 
For path-following and trajectory tracking of autonomous robots under
uncertainty we refer the reader to~\cite{AH07}.
Planning using composition of lower-level motion primitives has been studied,
among others, by~\cite{LeNP12,FDF05,FDF00,belta2007symbolic}. 
A general modeling framework for specifying hybrid systems is provided by hybrid
automata~\cite{Alur95thealgorithmic,ACHH92}.  
Given the expressiveness of hybrid automata, it is not surprising that simple
verification questions like reachability are undecidable~\cite{HKPV98} for
the general class of hybrid automata.
Given this result, there has been a growing body of work on decidable subclasses
of hybrid automata~\cite{ACHH92,BBM98}.  
Most notable among these classes are initialized rectangular hybrid
automata~\cite{HKPV98}, piecewise-constant derivative
systems~\cite{AMP95}, timed automata~\cite{alurDill94}, and multi-mode systems.

As mentioned earlier, stochastic multi-mode systems are a
generalization of constant-rate multi-mode systems. 
Alur, Trivedi, and Wojtczak~\cite{ATW12} considered constant-rate multi-mode
systems and showed that the reachability problem---deciding the reachability of
a specified state while staying in a given safety set---and the schedulability
problem---deciding the existence of a non-Zeno control so that the system
always stays in a given bounded and convex safety set---for this class of systems
can be solved in polynomial time.
  
Alur \emph{et al}.~\cite{AFMT13} introduced bounded-rate multi-mode
systems where the
rate in each mode is a constant that is picked from a given bounded
set.
These systems can be considered as constant-rate
multi-mode systems with uncertainties.
It is known~\cite{AFMT13,BJKST15} that the schedulability and reachability
problems for bounded-rate multi-mode systems are, although intractable
(co-NP-complete), decidable.
To the best of our knowledge, there is no known result on stochastic extensions
of multi-mode systems.

The paper is organized as follows. 
We begin by reviewing necessary  background on probability theory in the next
section, followed by the problem formulation in Section~\ref{sec:problem}.
In Section~\ref{sec:one-dim} we prove our key theorem for a simpler setting of
one-dimensional stochastic multi-mode systems.
We treat the case of general multi-dimensional systems in Section~\ref{sec:general}.
Although, the proof for one-dimensional system follows from the proof for the
general case, the proof for one-dimensional case is different and much simpler. 
We provide the algorithms based on our theorem to solve motion planning problem
in Section~\ref{sec:complexity},
before concluding in Section~\ref{sec:conclusion} by discussing
potential future directions.


%% file: motionPlanning.tex
\begin{figure*}[t]
  \begin{center}
    \scalebox{0.5}{
    \begin{tikzpicture}
    \draw[fill=gray!30, drop shadow={ashadow}] (0,6) -- (0,0) -- (2,4) -- cycle;
    \draw[fill=gray!40, drop shadow={ashadow}] (5,0) -- (10,1) -- (10,3) -- cycle;
    \draw[fill=gray!50, drop shadow={ashadow}] (3,6) -- (3,5) -- (6,1) -- (6,2)
    -- cycle;
    \draw[fill=gray!60, drop shadow={ashadow}] (8,7) -- (13,7) -- (11,4) -- (6,4)
    -- cycle;
    \draw[fill=gray!70, drop shadow={ashadow}] (12,4) -- (11,0) -- (12,0) -- cycle;
    \draw[fill=gray!80, drop shadow={ashadow}] (13,4) -- (15,4) -- (15,8) -- cycle;
    

    \draw[thick, color=black!50] (0, 0) rectangle (16cm, 8cm);
    \Tube{4mm}{25}{white}{black}
         {(2,1)
           to  (3,4)
           to (2,7)
           to (4, 7)
           to (6, 3)
           to (9, 3)
           to (11, 3.5)
           to (11, 2)
         }

   \draw node (r1) at (2, 1) {};
   \draw node (r2) at (11, 2){};
   \draw[fill=black] (r1) circle(0.05);
   \draw[fill=black] (r2) circle(0.05);
   
   \node[inner sep=0pt] (bb8) at (3,1)
       {\includegraphics[width=.08\textwidth]{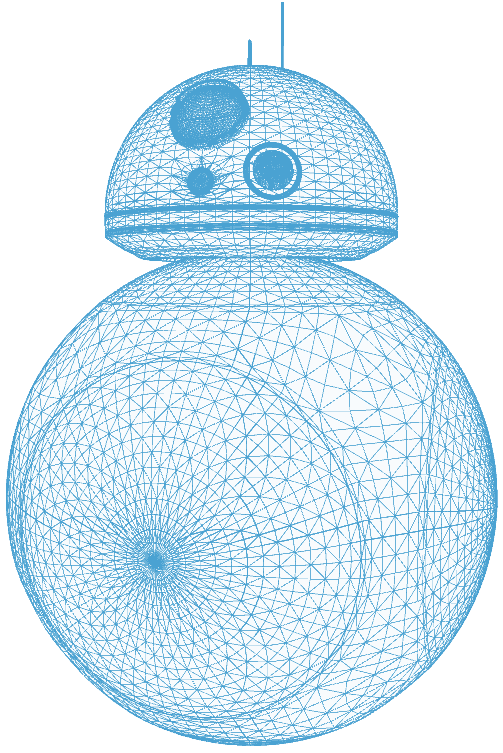}};
    
    \end{tikzpicture}
    }
    \hspace{4em}
  {
    \begin{tikzpicture}[scale=0.45]
    \tikzstyle{lines}=[draw=black!30,rounded corners]
    \tikzstyle{vectors}=[-latex, rounded corners]
    \tikzstyle{rvectors}=[-latex,very thick, rounded corners]

    \draw[lines] (4.2,0)--(10.2,0);
    \draw[lines] (7, 3.2)--(7,-3);
    
    \draw[vectors] (7, 0) --node[right]{$\phantom{x}\bar{\mu}_1$} (8.5, 1.5) node[left]{$$};
    \draw[fill=black,opacity=0.2] (8.5,1.5) circle (3cm);
    
    \draw[vectors] (7, 0) --node[right]{$\phantom{x}\bar{\mu}_2$} (7.5, -2.5) node[right]{$$};
    \draw[fill=orange,opacity=0.2] (7.5,-2.5) circle (3cm);
    
    \draw[vectors] (7, 0) --node[left]{$\bar{\mu}_3\phantom{xxxxx}$} (4.5, 0.5)node[right]{$$};
    \draw[fill=blue,opacity=0.2] (4.5, 0.5) circle (4cm); 

    \node[inner sep=0pt] (bb8) at (7, 0)
         {\includegraphics[width=.03\textwidth]{bb8.png}};

  \end{tikzpicture}
}
  \caption{Path planning for robots with high level of uncertainty.}
\label{fig:intro-example}
\end{center}
\end{figure*}


%% file: prelims.tex
Let $\R$ be the set of real numbers.
We write $[m]$ for the set $\{1,\ldots,m\}$.
For vectors $u,v\in \R^n$, we write $u\cdot v$ for the inner product of $u$ and
$v$, i.e., $u\cdot v:=\sum_{i=1}^nu_iv_i$.
We use $\|\cdot\|$ to denote the standard Euclidean-norm in $\Rn$,
i.e., $\|u\|:=(\sum_{i=1}^nu_i^2)^{1/2}$.
We say that a set $S\subseteq\Rn$ is bounded, if there exists a $\rho\geq 0$
such that $\|x\|\leq \rho$ for all $x\in S$.
For a vector $x\in \Rn$ and $d>0$, we define the ball $B(x, d)$ of radius $d$
around $x$ as
\[
B(x,d):=\{y\in\Rn\mid \|x-y\|<d\}.
\]
We also let the $\ell_1$ norm of $u$ be $\|u\|_1:=\sum_{i=1}^n|u_i|$. 
 
\subsection{Probability Space}
Here, we present background material and the main results in probability theory that will be needed later. Readers are referred to standard references such as \cite{durrett2010probability} for more details. Let $(\Omega,\F,P(\cdot))$ be a \emph{probability space} where $\Omega$ is a sample
space, $\F$ is a $\sigma$-algebra containing all the events of interest in the probability
space, and $P:\F\to[0,1]$ is a probability measure on $(\Omega,\F)$. For a given
probability space, we say that a property $p$ holds with probability 1 or almost
surely (a.s.) if  
\[
P(\{\omega\in \Omega\mid \omega\text{ satisfies }p\})=1.
\]
We assume all the probability measures in $\Rn$ discussed in this paper to be
Borel measures.

We say that a distribution (probability measure)
$\mu$ over $\R^n$ is \emph{compactly supported} if $\mu(\{x\in\Rn\mid \|x\|\geq
\rho\})=0$ for some $\rho>0$.
For a distribution $\mu$ over $\R^n$,  we write $\overline{\mu}$ to denote its
expected vector $E(\mu)$, which is the expected vector of a
random vector $Z$ whose distribution is $\mu$.
Moreover, for an event $A\in \F$, we write $1_A(\omega)$ for its characteristic
function, defined as 
\[1_A(\omega):=\left\{\begin{array}{ll}
1&\mbox{if $\omega\in A$}\cr 
0&\mbox{otherwise}
\end{array}\right..\]
For two random variables $X$ and $Y$, we use the \emph{wedge}-notation to denote
their minimum, i.e., $X\wedge Y:=\min(X,Y)$.

For an ensemble of random variables $\{\xi_k\}_{k\in I}$ over $(\Omega,\F)$, we
define $\sigma(\{\xi_k\}_{k\in I})$ to be the smallest $\sigma$-algebra such
that $\xi_k$s are measurable with respect to it.  
We also use the notation $\{y(k,\omega)\}_{k\geq 0}$ to denote a fixed sample
path $\omega\in \Omega$ of  a (discrete-time) random processes $\{y(k)\}$.  

\subsection{Martingales}
A \emph{filtration} $\{\F_k\}$ in a probability space $(\Omega,\F,P(\cdot))$ is a
sequence of sub-$\sigma$-algebras (of $\F$) such that
$\F_1\subseteq\F_2\subseteq \cdots$.
Let $\{\alpha(k)\}$ be a random process
with $E(|\alpha(k)|)<\infty$. We say that $\{\alpha(k)\}$ is adapted to
$\{\F(k)\}$, if $\alpha(k)$ is measurable with respect to $\F_k$ for all $k\geq
0$.

\begin{definition}[Martingales]
  We say that a random process $\{\alpha(k)\}$ adapted to a filtration
  $\{\F_k\}$ is a \textbf{martingale} with respect to a filtration $\{\F(k)\}$ if 
   \[E(\alpha(k+1)\mid \F_k)=\alpha(k),\]
 It is a \textbf{submartingale} if
   \[E(\alpha(k+1)\mid \F_k)\geq \alpha(k),\]
 and a \textbf{supermartingale} if 
   \[E(\alpha(k+1)\mid \F_k)\leq \alpha(k).\]
\end{definition}
We will make use of the following important result in martingale theory
(cf.\ Theorem~5.2.8 in \cite{durrett2010probability}).  

\begin{theorem}[Martingale convergence theorem]
  \label{thrm:MCT}
  ~Let\linebreak $\{\alpha(k)\}$ be a martingale such that $E(|\alpha(k))|)<B$
  for all $k\geq 0$ and some bound $B\in \R$. Then, 
  \[
  \alpha=\lim_{k\to\infty}\alpha(k)
  \]
  exists almost surely and $E(|\alpha|)<\infty$.   
\end{theorem}

Throughout this work, we assume that $\{\F_k\}$ is the natural filtration
for the underlying process $\{\alpha(k)\}$ and hence,  
\[
E(\alpha(k+1)\mid \F_k)=E(\alpha(k+1)\mid y(0),y(1),\ldots, y(k)).
\]

One of the immediate consequences of the martingale convergence theorem is the
following result that follows from the Robbins-Siegmund Theorem (cf.\ Theorem 7.11 in
\cite{Poznyak2009}).  
\begin{corollary}\label{cor:finitesum}
  If $\{\alpha(k)\}$ is a submartingale such that
  for all $k\geq 0$ we have $E(|\alpha(k)|)<B$ and  
  \begin{align}
    E(\alpha(k+1)\mid \F_k)\geq \alpha(k)+\xi(k),
  \end{align}
  where $\xi(k)\geq 0$ almost surely, then
  \[
  \sum_{k=0}^\infty \xi(k)<\infty \text{ ~almost surely.}
  \]
\end{corollary}

In addition to the above theorem, we make use of the following
result, which follows immediately from the definition of martingale
and the dominated convergence theorem (cf.\ Theorem~2.24 in
\cite{folland2013real}).
\begin{lemma}\label{lemma:stoppingtime}
  Let $\{\alpha(k)\}$ be a uniformly bounded supermartingale, i.e.,
  $|\alpha(k)|\leq B$ almost surely for some real number $B>0$ and all $k\geq
  0$. Then, if
  \[
  \alpha=\lim_{k\to\infty}\alpha(k),
  \]
  then we have 
  \[E(\alpha)\leq \lim_{k\to\infty}E(\alpha(k)).\]
  The same result holds for a uniformly bounded submartingale with the
  direction of the inequality reversed, i.e.,
  \[E(\alpha)\geq \lim_{k\to\infty}E(\alpha(k)).\]
  In particular, if $\{\alpha(k)\}$ is a uniformly bounded martingale,
  then $\lim_{k\to\infty}E(\alpha(k))=E(\alpha)$.  
\end{lemma}


%% file: problem.tex
An $n$-dimensional \emph{stochastic multi-mode system} (\smms{}) $\M$ is a
plant that is governed by a set of \emph{stochastic modes} in $\Rn$, i.e., a
set of distributions on $\Rn$, $\{\mu_1,\ldots,\mu_{\gamma}\}$ for
$\gamma \geq  1$.
The plant dynamics starts at a point $x(0)\in\Rn$ at the time
$t_0=0$.  At each discrete iteration  $k=0,1,\ldots$, the controller
chooses a mode $m(k)\in [\gamma]$ and dwelling time $d(k)>0$ for that
mode and then, the plant's dynamics follows 
\begin{align}
  \label{eqn:mmsdyn}
  \dot{x}(t)=\eta_{k+1}, \mbox{for $t\in [t_k,t_k+d(k))$} ,
\end{align}
where $\eta_{k+1}\in\Rn$ is a sample point from the chosen distribution
$\mu_{m(k)}$ and $t_k$ is recursively defined by
\begin{eqnarray*}
  t_k=
  \begin{cases}
    0 & \text{ if $k = 0$} \\ 
    t_{k-1} + d(k-1) & \text{ for $k>0$}
  \end{cases}.
\end{eqnarray*}
In other words, $x(t)=x(t_k)+(t-t_k)\eta_{k+1}$ for
$t\in [t_k,t_{k+1})$.  When the dimension $n$ is clear form the
context, we simply refer to a stochastic multi-mode system as a set
$\M = \{\mu_1,\ldots,\mu_{\gamma}\}$ of distributions.  An \smms{} is
\emph{deterministic} if for each mode $\mu\in \M$, we have
$P(\mu=v)=1$ for some (deterministic) vector $v\in \mathbb{R}^n$.
We denote the state of the plant at the decision times $t_0,t_1,\ldots$ by 
\begin{align}\label{eqn:location}
  y(k):=x(t_k), \mbox{ for }k=0,1,\ldots.
\end{align}
For the control of an \smms{}, our focus is to determine the mode $m(k)\in [\gamma]$ and the time $d(k)$ based
on the past observation of the system's behavior.
\begin{definition}[Control policy]
  \label{def:causal}
  A \emph{control policy} (or \emph{control strategy}) is
  a sequence $\{m(k),d(k)\}_{k\geq 0}$, where $m(k)$ is a random
  variable supported in $[\gamma]$ and $d(k)$ is a non-negative random
  variable for all $k\in \mathbb{Z}^+$. A control policy is \emph{causal} if
  $m(k)$ and $d(k)$ are measurable with respect to
  $\F_k=\sigma(\eta_1,\ldots,\eta_{k})$ ($\F_0=\{\emptyset,\Omega\}$)
  for all $k\geq 0$.
\end{definition}

Throughout this work, all control policies of interest are assumed
to be causal. The main property that we investigate is the
following reachability property.  
\begin{definition}[$\epsilon$-Reachability with safety set $S$]
  For a \smms{} $\M$ and a safety set $S \subseteq \Rn$, we say that
  $\M$ satisfies the $\epsilon$-reachability property with safety set
  $S$, or simply almost-sure reachability, if for any starting point $x_s \in S$, any terminal point
  $x_t \in S$, and any $\epsilon$-neighborhood of $x_t$, there exists a
  causal policy (controller) that steers $\M$ from the initial
  point $x(0)=x_s$ to a target point $x(T)=y(k)\in B(x_t,\e)$ in a
  finite time $T=t_k<\infty$ such that $x(t)\in S$ for all $t\in[0,T]$
  almost surely.
\end{definition}
Our approach to solve the $\varepsilon$-reachability problem is to characterize
necessary and sufficient conditions that guarantee
$\varepsilon$-reachability as described below.
\begin{problem}
  Given an open set $S\in \Rn$, under what conditions on $S$ and the stochastic
  modes can one guarantee $\epsilon$-reachability with safety set $S$?
\end{problem}
For a bounded safety set $S$, one can only use compactly supported
measures $\mu$ to maintain safety, as otherwise there is a non-zero probability
that the dynamics~\eqref{eqn:mmsdyn} does not satisfy safety. Therefore we
henceforth assume that all the distributions in $\M$ are compactly
supported and we let
\begin{equation*}
  L_{\M} = \max_i \inf \{ \rho \mid \mu_i(\{ x \in \Rn \mid \| x \|
  \geq \rho \}) = 0 \} .
\end{equation*}

Our key observation is that the $\epsilon$-reachability property of a stochastic
multi-mode system is closely related to the $\epsilon$-reachability
property of the associated (deterministic) expected multi-mode system defined
as follows. 
\begin{definition}[Expected Multi-mode System]
  The expected multi-mode system $\overline{\M}$ of 
  a stochastic multi-mode system $\M=\{\mu_1,\ldots,\mu_{\gamma}\}$
  is the deterministic multi-mode system whose
  dynamics in each mode is given by the expected direction of the corresponding mode in $\M$,
  i.e., $\overline{\M}$ is the deterministic multi-mode system with 
  modes $\overline{\M}=\{\overline{\mu}_1,\ldots,\overline{\mu}_{\gamma}\}$. 
\end{definition}
The main contribution of this work is the following result.
\begin{theorem*}
  Let $\M=\{\mu_1,\ldots,\mu_{\gamma}\}$ be a stochastic multi-mode system with
  a finite set of compactly supported distributions and let $S\subset \R^n$ be a
  path-connected and bounded open safety set.
  The following statements are equivalent:  
  \begin{enumerate}[a.]
  \item \label{item:eas}
    $\M$ satisfies the $\epsilon$-reachability property
    with safety set $S$.
  \item \label{item:ee}
    The expected multi-mode system $\overline{\M}$ of $\M$ satisfies
    the $\e$-reachability property with safety set $S$.
  \item \label{item:eav}
    For every non-zero vector $v \in \R^n$, there exists
    a mode $\mu \in \M$ such that $\overline{\mu} \cdot v>0$.
  \end{enumerate}
\end{theorem*}
In the next section, we visit this theorem in the context of $1$-dimensional
stochastic multi-mode systems and give a simpler proof for this result
than that of an $n$-dimensional \smms{}. Also, the statement and the
proof of the main result for the $1$-dimensional dynamics sheds light
on the statement and the proof of the main theorem in its complete
generality.


%% file: one-dim.tex
For the rest of this section we assume that the given stochastic multi-mode
system $\M = \set{\mu_1, \mu_2, \ldots, \mu_\gamma}$ is a $1$-dimensional
system. 
Before we characterize the necessary and sufficient condition for the
$\e$-reachability, we establish the following result.   
\begin{lemma}\label{lemma:Elog}
  Let $X$ be a random variable with $E(X)>0$ and $|X|\leq 1$ almost surely.
  Then we have that
  \[
  E(\log(1+\delta X))>0,
  \]
  for every $\delta<E(X)/2$.
\end{lemma}
\begin{proof}
  Let $\delta< E(X)/2\leq 1/2$. For $z\in [-\delta,\delta]$, the
  Taylor expansion of $f(z)=\log(1+z)$ around $z=0$ implies: 
  \begin{align*}
    \log(1+z)&=z-\frac{1}{(1+\tilde{z})^2}\frac{z^2}{2}\geq
    z-\frac{1}{(1-\delta)^2}\frac{z^2}{2}\geq z-2z^2, 
  \end{align*}
  for some $\tilde{z}\in[-\delta,\delta]$, where the first inequality follows
  from the fact that $1/(1+\tilde{z})^2$ is a decreasing function of
  $\tilde{z}$ and the second inequality follows from the fact that  
  \[\delta\leq E(X)/2\leq \frac{1}{2}.\] 
  Therefore, for a random variable $X$ whose support is in $[-1,1]$, we have:
  \begin{align*}
    E(\log(1+\delta X)) &\geq \delta E(X)-2\delta^2E(X^2) \\
    &\geq 2\delta(E(X)/2-\delta)> 0,
  \end{align*}
  where the last inequality follows from $\delta\in (0,E(X)/2)$.
\end{proof}

Next, we show that almost sure $\e$-reachability with safety set
$S=(a,b)$ is achievable if and only if there exist two modes $\mu_{+}$
and $\mu_{-}$ such that $\bar{\mu}_{+}>0$ and $\bar{\mu}_{-}<0$. Note
that for a deterministic system this is indeed necessary and
sufficient: it is necessary because otherwise one cannot move the
deterministic system from $x_s=\frac{3a+b}{4}$ to $x_t=\frac{a+3b}{4}$
or vice versa. It is sufficient because, once we have a positive and a
negative control direction, then one can steer the system
towards left and right to the desired position without violating
safety. However, to show such a result for stochastic systems one must
account for the possibly adversarial effects of the noise in the
control vectors. Through a proper choice of control policy we want
to make sure that the system will reach the target while ensuring
safety.

\begin{theorem}\label{thrm:1dcase}
  Let $\M=\{\mu_1,\ldots,\mu_{\gamma}\}$ be a \smms{} with 
  a finite set of compactly supported distributions and, without loss of
  generality, assume that each $\mu_i \in \M$ is distributed over the
  unit interval. 
  Further assume that the safety set is $S=(a,b)$ for some $a<b$.

  Then, $\M$ satisfies the $\epsilon$-reachability property with safety set $S$
  if and only if there exist modes $\mp\in \M$ and $\mm\in \M$ such that
  $E(\mp)>0$ and $E(\mu-)<0$.
  
  In other words, $\M$ satisfies the $\epsilon$-reachability property with
  safety set $S$ if and only if the expected deterministic 
  system $\overline{\M}$ satisfies the same property.
\end{theorem}
\begin{proof}
  Without loss of generality, we assume that $a=0$.

  (Necessity.)  Suppose that $\M$ does not have a mode $\mm$ with
  $\overline{\mu}_{-}<0$. We prove a stronger statement: we show that there does not
  exists any causal policy that can almost surely reach the
  $\epsilon$-neighborhood of $x_t$ for any starting point $x_s\in (a,b)$, any
  target point $x_t\in (a,b)$ with $x_t<x_s$ and any $\epsilon<(x_s-x_t)$.

  To show  this consider any causal policy $\{m(k),d(k)\}_{k\geq 0}$ that
  guarantees safety almost surely. Consider the random process $\{y(k)\}$
  defined by \eqref{eqn:location}. Note that $y(k)$ is adapted to $\F_k$ (the
  natural filtration for the $\eta_k$). It follows that:  
  \begin{align*}
    E(y(k+1)\mid \F_k)=y(k)+d(k+1)E(\mu_{m(k)})\geq y(k).
  \end{align*}
  Therefore, $\{y(k)\}$ is a submartingale w.r.t.\ the filtration
  $\{\F_k\}$. Also, since the policy guarantees almost sure safety and $S$ is
  bounded,  $\{y(k)\}$ is a bounded martingale and hence by
  Theorem~\ref{thrm:MCT}, it is convergent almost surely.
  Let us define the stopping time $T$ as follows: 
  \begin{align*}
    T:=\inf\{k\geq 0\mid y(k)\in B(x_t,\epsilon)\}.
  \end{align*} 
  Since $\{y(k)\}$ is a bounded submartingale, $\{y(k\wedge T)\}$ would be a
  bounded submartingale and it is convergent almost surely, i.e.,  
  \begin{align}
    y=\lim_{k\to\infty}y(k \wedge T)
  \end{align}
  exists almost surely.
  Hence by Lemma~\ref{lemma:stoppingtime} we get that
  \[E(y)=\lim_{k\to\infty}E(y(k\wedge T))\geq E(y(0))=x_s.\]
  In particular, if we let $p=P(y\in B(x_t,\epsilon))$, then 
  \[x_s\leq E(y)\leq (x_t+\epsilon)p+(1-p)b,\]
  and hence 
  \[p\leq \frac{b-x_s}{b-(x_t+\epsilon)}<1.\]
  Therefore almost-sure convergence is impossible.  
	
  The impossibility of almost sure $\e$-reachability for the case that \smms{}
  does not contain a mode $\mp$ with $E(\mp)>0$ follows from the same argument
  presented above.
	
  (Sufficiency.) To show the sufficiency part of the theorem, let $x_s\in
  (a,b)$ be any given 
  initial condition and $x_t\in (a,b)$ be any given target point and let
  $\epsilon>0$. We assume that $x_t+\epsilon<b$. If this condition is not met,
  we replace $\epsilon$ with $\tilde{\e}=(b-x_t)$ in the following
  argument (note that if $x(t)$ is in $B(x_t,\tilde{\e})$ it will also belong to
  $B(x_t,\e)$ as $B(x_t,\tilde{\e})\subseteq B(x_t,{\e})$).
  
   Recall that without loss of
  generality we may assume that $a=0$ (so that the length of $S$ is $b$) and
  $x_s<x_t$. Consider the constant control policy for mode selection $\{\mp\}$
  and the dwelling time sequence $\{d(k)\}$ be as follows: 
  \begin{align}\label{eqn:cp1d}	
    d(k)=\delta y(k-1),
  \end{align}
  for a constant $\delta<\min(\overline{\mu}_{+}/2,\epsilon/(b-a))$. We show that
  $y(k)$ enters $B(x_t,\epsilon)$ with probability 1 while maintaining
  safety. One can show inductively that almost surely the policy
  \eqref{eqn:cp1d} guarantees $x(t)>0$ for all $t\geq 0$. Also, note that  
  \begin{align*}
    y(k) &=y(k-1)+d(k)Z(k) \\
    &\leq y(k-1)+\frac{\epsilon}{b-a}Z(k)\leq y(k-1)+\epsilon,
  \end{align*} 
  where $Z(k)$ is sampled from $\mp$.  Note that 
  \begin{align*}
    y(k)
    &= y(k-1)+\delta y(k-1)Z(k) \\
    &= y(k-1)(1+\delta Z(k)) = y(0)\prod_{i=1}^{k}(1+\delta Z(i)),
  \end{align*}
  and therefore, 
  \begin{align}
    \log(y(k))=\log(y(0))+\sum_{i=1}^k\log(1+\delta Z(i)).
  \end{align}
  
  Note that  $Z(i)$-s are i.i.d.\ random variables and hence $\log(1+\delta
  Z(i))$-s are i.i.d.\ random variables.
  By Lemma~\ref{lemma:Elog} and the choice of $\delta$, it follows that
  \[
  E(\log(1+\delta Z(i)))>0.
  \]
  Therefore, invoking the
  \emph{strong law of large numbers} (cf. Theorem
  2.4.1.\ \cite{durrett2010probability}), it follows that
  \[
  \lim_{k\to\infty}\log(y(k))=\infty
  \]
  and we have $y(k)>x_t$ with probability one for some $k$ (depending on
  $\omega\in \Omega$). Finally, since
  \[
  y(k)\leq  y(k-1)+\epsilon,
  \]
  therefore, almost surely  $y(k)\in B(x_t,\epsilon)$
  for some $k(\omega)$.  
The proof is now complete.  
\end{proof}


%% file: general.tex
In this section, we prove the
extension of Theorem~\ref{thrm:1dcase} for an arbitrary open, bounded, and
path-connected safety set in $\R^n$.  Namely, we show that in order
to drive a system from any starting point to any target point, for any
direction $v$ there must exist a stochastic mode $\mu$ such that we can positively move along $v$ in expectation using $\mu$, i.e., $v\cdot
\bar{\mu}>0$.
So, the main result of this section is as follows. 
\begin{theorem}\label{thrm:mdcase}
  Let $\M=\{\mu_1,\ldots,\mu_{\gamma}\}$ be a stochastic multi-mode system with
  a finite set of compactly supported distributions and let $S\subset \R^n$ be a
  path-connected and bounded open safety set.
  The following statements are equivalent:  
  \begin{enumerate}[a.]
  \item \label{item:eas}
    $\M$ satisfies the $\epsilon$-reachability property
    with safety set $S$.
  \item \label{item:ee}
    The expected multi-mode system $\overline{\M}$ of $\M$ satisfies
    the $\e$-reachability property with safety set $S$.
  \item \label{item:eav}
    For every non-zero vector $v \in \R^n$, there exists
    a mode $\mu \in \M$ such that $\overline{\mu} \cdot v>0$.
  \end{enumerate}
\end{theorem}
\begin{figure*}[ht]
	\def\myscale{0.8}
	\centering
	\subfloat[]{\label{fig:movecentere}
		\begin{tikzpicture}[->,>=stealth',shorten >=1pt,auto,node distance=0.5cm,
		semithick,scale=\myscale,every node/.style={transform shape}]
		\draw[] (0, 0) circle (2.5cm);
		\draw node (r1l) at (0.2,0.2) {$x_0$};
		\draw node (r2l) at (-0.25, -1.15) {$y_k$};
		\draw[-,dotted] (0, -1) -- node[below] {$\sqrt{\alpha(k)}$} (2.3, -1);
		\draw[dotted,->] (0, 0) -- node[right, pos=0.5] {$~~r$} (2.3, -1);
		\draw node (r1) at (0, 0) {};
		\draw node (r2) at (0, -1){};
		\draw[fill=black] (r1) circle(0.05);
		\draw[fill=black] (r2) circle(0.05);
		\draw[dotted] (0, 0) circle (0.5cm);
		
		\draw[-,dotted] (-3, 0) -- (3, 0);
		\draw[-,dotted] (0, -2.8) -- (0, 2.8);
		\end{tikzpicture}}%
	\subfloat[]{\label{fig:movecircle}
		\begin{tikzpicture}[->,>=stealth',shorten >=1pt,auto,node distance=0.5cm,
		semithick,scale=\myscale,every node/.style={transform shape}]
		\def\ra{0.4}
		\draw[] (0, 0) circle  (2.5cm);
		\draw node (r1l) at (-1.6, -0.7) {$x_s$};
		
		\draw node (r1) at (-1.6, -0.2) {};
		\draw[fill=black] (r1) circle(0.05);
		\draw[dashed] (r1) circle(\ra);
		
		\draw node (rdl) at (1.6, -0.7) {$x_t$};
		\draw node (rd) at (1.6,-0.2){};
		\draw[fill=black] (rd) circle(0.05);
		\draw[dashed] (rd) circle(\ra);

		\foreach \i in {0,...,8}
		{
			\draw node (\i) at (-1.6+\ra*\i, -0.2) {$u_{\i}$};
			\draw[fill=gray!30,color=gray] (\i) circle(0.05);
			\draw[dashed,color=gray] (\i) circle(\ra);
		}
		\draw[-,opacity=0.5,dotted] (-3, 0) -- (3, 0);
		\draw[-,opacity=0.5,dotted] (0, -2.8) -- (0, 2.8);
		\end{tikzpicture}}
	\subfloat[]{\label{fig:necessity}
		\begin{tikzpicture}[->,>=stealth',shorten >=1pt,auto,node distance=0.5cm,
		semithick,scale=\myscale,every node/.style={transform shape}]
		\draw[-, color=black!20, dashed] (0,0) -- (-2.5,-2.5);
		\draw[->, color=black, thick] (0,0) -- (2.5,2.5);
		
		\draw[-,decorate,decoration={brace,amplitude=4pt,mirror,raise=1pt},yshift=0pt] (0.4,-0.4) -- node[below] {$~~~~~~\frac{\delta}{2}$} (1.5, 0.7);
		
		\draw[-,dotted] (-3, 0) -- (3, 0);
		\draw[-,dotted] (0, -2.8) -- (0, 2.8);
		
		\draw[->] (0, 0) -- (-2, -0.5);
		\draw[-,thick,color=black!30] (0, 0) -- node [color=black,right, pos=0.8] {$\alpha(k)$} (-1.3, -1.3);
		
		\draw node (r1l) at (-0.7,0.2) {$y(0)= x_s$};
		
		\draw node (yll) at (-2, -0.3) {$y(k)$};
		\draw node (y1) at (-2, -0.5) {};
		\draw node (r2l) at (1.2, 0.9) {$x_t$};
		\draw node (vl) at (2, 1.5) {$\vec{v}$};
		\draw node (r1) at (0, 0) {};
		\draw node (r2) at (1, 1){};
		\draw[fill=black] (r1) circle(0.05);
		\draw[fill=black] (r2) circle(0.05);
		\draw[dotted] (1, 1) circle (0.6cm);
		\draw[dotted,->] (1, 1) -- node[right, pos=0.7] {$~\frac{\delta}{4}$} (0.5, 1.5);
		\end{tikzpicture}}
	\caption{Stochastic multi-mode system: proof sketch.\label{fig:smms-proof}}
\end{figure*}
As for 1-dimensional dynamics, we prove \ref{item:eas} $\Rightarrow$
\ref{item:eav} by a martingale argument. For the converse,
we break the problem into sub-problems: we show that if condition
\ref{item:eav} holds, then  
\begin{enumerate}[i.]
\item If the safety set $S$ is a ball, the controller can reach the
  $\epsilon$-neighborhood of the center of the ball from any starting point
  $x_s\in S$ (Lemma~\ref{lemma:centerball}). 
\item If the safety set $S$ is a ball, the controller can reach the
  $\epsilon$-neighborhood of any point $x_t\in S$ from any starting point
  $x_s\in S$ (Lemma~\ref{lemma:Sball}). 
\item For any path-connected open set $S$, the controller can traverse
  from any starting
  point $x_s$ to any $\e$-neighborhood of any target point $x_t\in S$ by moving
  inside a sequence of balls that are strictly within the safety set.
\end{enumerate}
We proceed by formulating and proving the above intermediate steps.

 \begin{lemma}\label{lemma:centerball}
   Let $\M=\{\mu_1,\ldots,\mu_{\gamma}\}$ be a stochastic multi-mode system with 
   a finite set of compactly supported distributions and let the safety set $S$ be a ball
   $B(x_0,r)\subset \R^n$. 
   If condition \ref{item:eav} of Theorem~\ref{thrm:mdcase} holds then, for
   any starting point $x_s\in S$ and any $\epsilon>0$, the system can reach
   the $\e$-neighborhood of the center $x_t=x_0$ almost surely.
\end{lemma}
\begin{proof}
  Let $\M=\{\mu_1,\ldots,\mu_{\gamma}\}$ be a stochastic multi-mode system with
  a finite set of compactly supported distributions.  Suppose that the safety
  set $S$ is  $B(x_0,r)\subset \R^n$ and w.l.o.g.\ assume that
  $x_0=0$ (see, Figure~\ref{fig:movecentere}).
  Assume that $x_t=x_0=0$ (i.e., $x_t$ is the center of the safety ball) and
  assume condition \ref{item:eav} of Theorem~\ref{thrm:mdcase} holds.
  We show that for any starting point $x_s\in S$ and any $\epsilon>0$, there is
  a control policy to steer the system to the $\e$-neighborhood of $x_0$ almost
  surely while staying within the safety set $S$. 

  Without loss of generality assume that $\e\leq r$. Let $L$ be
  $L_{\M}$, the maximum support of the modes in $\M$. We use the 
  following controller policy for the $\e$-reachability of $x_t$:\\ 
  At iteration $k$, if $y(k)\not \in B(0,\epsilon)$, we let:  
  \begin{align}\label{eqn:policymd}
    m(k)&\in \argmax_{i\in[m]}\overline{\mu}_i\cdot (-y(k)),\cr 
    d(k)&={\delta}(r^2-\|y(k)\|^2),
  \end{align}
  for a sufficiently small positive constant $\delta$ satisfying
  \begin{align}\label{eqn:delta1}
    \delta< \min(\frac{1}{L},\frac{1}{2Lr}).
  \end{align} 
  We require $\delta$ to satisfy further inequalities that will be discussed later. If $y(k)\in
  B(0,\epsilon)$, we simply let $d(k)=0$ (or in other words, we stop the
  process). Let us denote the event $y(k)\not \in B(0,\epsilon)$ by $A_k$.  
  
	If $\|y(k)\|<r$, then for any $\delta\in (0,1/(2Lr))$ and on $A_k$,
	almost surely we  have:  
	\begin{align*}
	\|y(k+1)\|&=\|y(k)+d(k)Z(k)\|\cr
	& \leq \|y(k)\|+d(k)\|Z(k)\|\cr 
	&< \|y(k)\|+\delta(r^2-\|y(k)\|^2)L \cr
	&= \|y(k)\|+\delta(r-\|y(k)\|)(r+\|y(k)\|)L\cr 
	&< \|y(k)\|+(r-\|y(k)\|)\delta 2rL\cr
	&< \|y(k)\|+(r-\|y(k)\|)= r,
	\end{align*}
	where the last inequality follows from $\delta<\frac{1}{2L}$ and $\|y(k)\|<r$. Hence, inductively, policy \eqref{eqn:policymd} satisfies safety
	almost surely.

	By expanding $\|y(k+1)\|^2=(y(k)+d(k)Z(k))\cdot (y(k)+d(k)Z(k))$, we get: 
	\begin{align*}
	\|y(k+1)\|^2&=\|y(k)\|^2+1_{A_k}\Big({\delta}^2(r^2-\|y(k)\|^2)^2\|Z(k)\|^2\cr
	&\qquad\qquad+2{\delta}(r^2-\|y(k)\|^2)y(k)\cdot Z(k)\Big) .
	\end{align*}
	Subtracting both sides of the above equality from $r^2$, we get:
	\begin{align*}
	&r^2-\|y(k+1)\|^2= (r^2-\|y(k)\|^2)\cr
	&\qquad -1_{A_k}(r^2-\|y(k)\|^2)\times \cr 
	&\qquad\qquad\left({\delta}^2(r^2-\|y(k)\|^2)\|Z(k)\|^2+2{\delta}y(k)\cdot
	Z(k)\right). 
	\end{align*}
	
	Letting $\alpha(k):=(r^2-\|y(k)\|^2)$ (see, Figure~\ref{fig:movecentere}), this equality simplifies to:
	\begin{align}\label{eqn:ykineq2}
	&\alpha(k+1)\cr 
	&= \alpha(k)(1-1_{A_k}({\delta}^2(r^2-\|y(k)\|^2)\|Z(k)\|^2+2{\delta}y(k)\cdot Z(k)))\cr 
	&\geq \alpha(k)\left(1-1_{A_k}({\delta}^2r^2L^2+2{\delta}y(k)\cdot Z(k))\right).
	\end{align}
	
	Note that $g(v)=v\cdot\overline{\mu}_i $ is a continuous functional on $\R^n$
	therefore, the function $v\to\max_{i\in\M}v\cdot \overline{\mu}_i$ is a
	continuous function. Since, the set $\{v\in \R^n\mid \|v\|=1\}$ is a compact
	set in $\R^n$, we get  
	\[\lambda:=\inf_{v\in \R^n: \|v\|=1}\max_{i\in[\gamma]}v\cdot \overline{\mu}_i>0.\]
	Let $p(k):=2y(k)\cdot Z(k)+\epsilon\lambda$. Then, \eqref{eqn:ykineq2} simplifies to: 
	\begin{align}\label{eqn:ykineq3}
	\alpha(k+1)&\geq  \alpha(k)(1+1_{A_k}(-{\delta}^2r^2L^2+\delta \epsilon\lambda-\delta p(k)).
	\end{align}	 
	Note that for 
	\begin{align}\label{eqn:delta2}
	\delta\leq \frac{\e\lambda}{r^2L^2},
	\end{align} 
	we have $-{\delta}^2r^2L^2+\delta \epsilon\lambda\geq  0$. 
	Therefore, for a $\delta$ satisfying \eqref{eqn:delta1} and \eqref{eqn:delta2}, we have: 
	\begin{align*}
	\alpha(k+1)&\geq \alpha(k)(1-\delta 1_{A_k}p(k)).
	\end{align*}	 
	Applying $\log(\cdot)$ on both sides of the above inequality, we get:
	\begin{align*}
	\log(\alpha(k+1))&\geq \log(\alpha(k))+1_{A_k}\log(1-\delta p(k)).
	\end{align*}	 
	Note that on $A_k$,
	\begin{align}\label{eqn:pkbound}
	|p(k)|\leq 2\|y(k)\|\|Z(k)\|+\epsilon\lambda\leq 2rL+\epsilon\lambda. 
	\end{align}
	Also, 
	\begin{align*}
	E(-p(k)\mid \F_k)&=(E(2(-y(k))\cdot Z(k)-\epsilon\lambda\mid \F_k))\cr 
	&\geq (2\|y(k)\|\lambda -\epsilon\lambda) 
	\geq (2\epsilon\lambda -\epsilon\lambda)=\epsilon\lambda>0,
	\end{align*}
	where the first inequality follows from the choice of $m(k)$ in \eqref{eqn:policymd}
	and the definition of $\lambda$.  
	Since $E(p(k)\mid \F_k)>0$ and $|p(k)|\leq 2rL+\epsilon\lambda$, by
	Lemma~\ref{lemma:Elog}, for  
	\begin{align}\label{eqn:delta3}
	\delta\leq \frac{\epsilon\lambda}{4rL+2\epsilon\lambda}
	\end{align}
	that also satisfies \eqref{eqn:delta1} and \eqref{eqn:delta2}, we have:
	\begin{align}\label{eqn:delta3}
	  E(\log(1 + & 1_{A_k}\delta p(k)) \mid \F_k) \cr
          &= 1_{A_k}E(\log(1+\delta p(k))\geq 1_{A_k}\xi, 
	\end{align} 
	for some $\xi>0$.
	Therefore, for such a small $\delta$, we have:
	\begin{align*}
	E(\log(\alpha(k+1))\mid \F_k)\geq \log(\alpha(k))+1_{A_k}\xi. 
	\end{align*}
	Since, $\alpha(k+1)$ is bounded, by Corollary~\ref{cor:finitesum}, it follows that 
	\[\sum_{k=0}^\infty 1_{A_k}<\infty,\] almost surely. 
	Therefore, almost surely, the trajectories of the dynamics will enter
	$B(x_t,\e)=B(0,\e)$. 
\end{proof}
Using this result, the next step is to show that the controller can achieve almost sure $\e$-reachability for the safety set being a ball (assuming the conditions of Lemma~\ref{lemma:centerball}). 

\begin{lemma}\label{lemma:Sball}
  Consider a stochastic multi-mode system with a finite set
  $\M=\{\mu_1,\ldots,\mu_{\gamma}\}$ of compactly supported 
  stochastic modes satisfying \eqref{item:eav}.
  If the safety set $S$ is a ball $B(x_0,r)\subset \R^n$, then the \smms{}
  satisfies the almost sure $\e$-reachability property with the safety set $S=B(x_0,r)$.   
\end{lemma}
\begin{proof}
  Let $x_s$ and $x_t$ be arbitrary starting and target points in the safety set
  $S=B(x_0,r)$ for some $r\geq \epsilon>0$.
  Let
  \[
  \tilde{r}=\frac{1}{2}\min(r-\|x_s-x_0\|,r-\|x_t-x_0\|).
  \]
  Note that for any point $v\in\{\beta x_s+(1-\beta)x_t\mid \beta\in[0,1]\}\subset
  S$, the segment connecting $x_s$ and $x_t$ in $S$, $B(v,2\tilde{r})\subset
  S$ holds.

  Let $N=\lfloor(\|x_t-x_s\|)/\tilde{r}\rfloor+1$ and let
  $u_i=x_s{+}(i/N){.}(x_t{-}x_s)$ for any $i=0,\ldots, N$, where
  $z=\lfloor\beta\rfloor$ is the largest integer that satisfies $z\leq
  \beta$. Note that $u_0=x_s$ and $u_N=x_t$.
  These intermediate points are illustrated  in Figure~\ref{fig:movecircle}.
  
Let $\tilde{\e}< \min(\e,\tilde{r})$. Then, for any $u\in B(u_i,\tilde{\e})$, we have:
\begin{align*}
\|u_{i+1}-u\|&=\|u_{i+1}-u_i+u_i-u\|\cr 
&\leq \|u_{i+1}-u_i\|+\|u_i-u\|\leq \tilde{r}+\tilde{\e}<2\tilde{r},
\end{align*}
for $i=0,\ldots, N$. By Lemma~\ref{lemma:centerball}, we can move from any
point in the $\tilde{\e}$-neighborhood of $u_i$ to a point in
$\tilde{\e}$-neighborhood of $u_{i+1}$ satisfying the safety
$B(u_{i+1},2\tilde{r})\subset S$. Therefore, by induction, almost surely, we
can traverse from $u_0=x_s$ to $B(u_N=x_t,\tilde{\e})\subseteq B(x_t,\e)$
while satisfying safety $S$ almost surely. 
\end{proof}

Finally, we are in a position to complete the proof of the main result. 

\begin{proof}[of Theorem~\ref{thrm:mdcase}]
  Let $\M=\{\mu_1,\ldots,\mu_{\gamma}\}$ be a sto\-chastic multi-mode system with
  finite set of compactly supported distributions and let $S\subset \R^n$ be a
  path-connected and bounded open safety set.
  It suffices to show the equivalence of \ref{item:eas} and
  \ref{item:eav} as the same equivalence holds for the deterministic
  system and the directions of its modes.
  In particular we show that 
  $\M$ satisfies almost sure $\epsilon$-reachability property
  with the safety set $S$ if and only if for any $v\in \R^n$ there exists
  a mode $\mu \in \M$ such that $\overline{\mu} \cdot v>0$.
  
  (\ref{item:eas} $\Rightarrow$ \ref{item:eav})
  Suppose that \ref{item:eas}
  holds but \ref{item:eav} does not. Let $v\not=0$ be a vector such that
  for every $\mu\in \M$, $\overline{\mu}\cdot v\leq 0$.
  If a vector $v$ satisfies such a property, then the unit-length vector
  $\frac{v}{\|v\|}$ also does. So, without loss of
  generality we assume that $\|v\|=1$.
  Let $x_s\in S$ be an arbitrary starting point.
  Since $S$ is an open set, there exists a $\delta>0$ such that
  $B(x_s,\delta)\subset S$.
  Let $x_t=x_s+\frac{\delta}{2}v$ and let $\epsilon=\frac{\delta}{4}$ (see Figure~\ref{fig:necessity}). Now,
  consider an arbitrary causal control policy $(m(k),d(k))$ and let $\{y(k)\}$ be
  defined as in \eqref{eqn:location} and $\{\F_k\}$ be the
  $\sigma$-algebra that is adapted to $\{y(k)\}$. Define
  $\alpha(k):=v\cdot y(k)$ for all $k\in \mathbb{Z}^+$. Note that
  $\{\alpha(k)\}$ is adapted to $\{\F_k\}$ and also,   
  \begin{align*}
    E(\alpha(k+1)\mid \F_{k}) &=E(v \cdot (y(k)+d(k)Z(k))\mid \F_{k})\\
    &=v \cdot y(k)+d(k)E(v \cdot Z(k)) \\
    &\leq v \cdot y(k)=\alpha(k),
  \end{align*}
  where $Z(k)$ is the random vector whose distribution is
  $\mu_{m(k)}$. Therefore, the sequence $\{\alpha(k)\}$ would be a
  supermartingale. Note that $S$ is bounded and hence, if the policy guarantees
  almost sure safety, this supermartingale is convergent to a random variable
  $\alpha$. Therefore,  
  \begin{align}\label{eqn:submar-md}
    E(\alpha)\leq E(\alpha(0))=v\cdot x_s.
  \end{align}
  On the other hand, for any $u\in B(x_t,\frac{\delta}{4})$, we have:
  \begin{align*}
    v\cdot u&=v\cdot (u-x_t+x_t)= v\cdot (u-x_t)+v\cdot x_t\cr 
    &\geq -\frac{\delta}{4}\|v\|^2+v\cdot (x_s+\frac{\delta}{2} v)=v\cdot x_s+\frac{
      \delta}{4}\|v\|^2 \\
    &>v\cdot x_s=E(\alpha(0)),
  \end{align*}
   where the inequality follows from the Cauchy-Schwartz inequality and the fact
  that $u-x_t\in B(0,\frac{\delta}{4})$. Therefore, if the sample paths almost
  surely reach $B(x_t,\frac{\delta}{4})$, we have $E(\alpha){>}E(\alpha(0))$ which
  contradicts \eqref{eqn:submar-md}.  
	
  (\ref{item:eav} $\Rightarrow$ \ref{item:eas}) Indeed this part of the Theorem applies for any path-connected open set $S$ (that is not necessarily bounded).  Let $S$ be an arbitrary path-connected and open set and let $x_s,x_t\in
    S$. Since $S$ is path-connected, there exists a continuous path $\nu:[0,1]\to
    S$ such that $\nu (0)=x_s$ and $\nu(1)=x_t$. For any $\theta\in [0,1]$, let
    $\rho_{\theta}>0$ be such that $B_{\theta}:=B(\theta,\rho_{\theta})\subset
    S$. Such $\rho_{\theta}$ exists because $S$ is an open set. Also, the image
    of $[0,1]$ under $\nu $, i.e., the set  
    \[\mathcal{C}:=\{\nu (\theta)\mid \theta\in [0,1]\} \]
    is a compact subset of $S$ as it is the image of a compact interval $[0,1]$
    under the continuous map $\nu$. Finally, the collection
    $I=\{B_{\theta}\}_{\theta\in[0,1]}$ is an open cover for $\mathcal{C}$,
    i.e.\ 
    \begin{align*}
      \mathcal{C}\subset \bigcup_{\theta\in[0,1]}B_{\theta}.
    \end{align*}
    By compactness of $\mathcal{C}$, there exists a finite open sub-cover of $I$
    that covers $\mathcal{C}$. In other words, there exists
    $\theta_1,\ldots,\theta_q\in [0,1]$ such that  
    \begin{align*}
      \mathcal{C}\subset B_{\theta_1}\cup B_{\theta_2}\cup \cdots\cup B_{\theta_q},
    \end{align*}
    for some finite number $q\in \mathbb{Z}^+$. Construct the undirected
    intersection graph $G=([q],E)$ where $[q]:=\{1,\ldots,q\}$ and  
    \[E=\left\{\{i,j\}\mid i,j\in[q], B_{\theta_i}\cap B_{\theta_j}\not=\emptyset\right\}.\]
    One can verify that the intersection graph $G$ should be connected, as otherwise, the set
    $\mathcal{C}$ would be a (path) disconnected set.  
    
    Without loss of generality assume that $x_s\in B_{\theta_1}$ and $x_t\in
    B_{\theta_q}$. Let $u_1=1\to u_2\to \cdots \to u_k=q$ be a directed path in
    $G$ that connects vertex $1$, which is associated with $B_{\theta_1}$, which
    contains $x_s$, to the vertex $q$, which is associated with $B_{\theta_q}$,
    which contains $x_t$. Now let $x_1=x_s$, $x_k=x_t$ and choose the points
    $x_i\in B_{\theta_{u_i}}\cap B_{\theta_{u_{i+1}}}$ for
    $i\in\{2,\ldots,k-1\}$. By Lemma~\ref{lemma:Sball}, there exists a control
    policy that starting from any starting point in 
    $B(x_i,\tilde{\e})\subset B_{\theta_{u_i}}$ the controller can move to some point in 
    $B(x_{i+1},\tilde{\e})\subset B_{\theta_{u_i}}$ while maintaining safety
    $B_{\theta_{u_i}}\subset S$ almost surely. Therefore, the controller can traverse from 
    $x_s$ to $x_t$ by concatenating these control policies while
    maintaining safety $S$ almost surely.  
  The proof is now complete.
\end{proof}


Note that the characterization in Theorem~\ref{thrm:mdcase} is
independent of the safety set as long as it is open, bounded, and
path-connected. Therefore, one may regard $\e$-reachability to be a
property of the \smms{} independent of the safety set (as long as the
latter satisfies those conditions).

Another observation about Theorem~\ref{thrm:mdcase} is that the bounded condition on
the safety set $S$ is absolutely required to prove that \ref{item:eav} is
necessary for $\e$-reachability. To show the importance of this condition let us
discuss a simple example.  
\begin{example}
  Consider the safety set $S=\mathbb{R}$ and the \smms{} $\M=\{\mu_1\}$ where
  \[
  P(\mu_1=+1)=P(\mu_1=-1)=\frac{1}{2}.
  \]
  Notice that this is the case of a simple random walk on $\mathbb{Z}$. 
  Consider the simple control policy $d(k)=\epsilon$ and $m(k)=1$ for all $k\geq 0$. 
  It can be shown that (cf.\ Theorem 4.1.2 in \cite{durrett2010probability})
  for any initial condition $y(0)=x(0)=x_s\in\mathbb{R}$:  
  \[\liminf_{k\to\infty}y(k)=-\infty \text{ and }  \limsup_{k\to\infty}y(k)=+\infty.\]
  Therefore, starting from any starting point $x_s$, this controller will almost surely visit the $\e$-neighborhood of 
  any target point $x_t\in \R$. Therefore, this \smms{} $\M$ satisfies almost reachability in $S=\R$. 
  
  However, the expected value of each mode is $\bar{\mu}_1=0$ which
  clearly does not satisfy \ref{item:eav}. The reason that
  Theorem~\ref{thrm:1dcase} and Theorem~\ref{thrm:mdcase} fail in this case is
  that the safety set is no longer a bounded set.  
\end{example}

Although the boundedness of the safety set $S$ is necessary to prove that \ref{item:eas} implies \ref{item:eav}, the proof of the reverse implication does not rely on the boundedness of the safety set. 

\begin{corollary}\label{cor:unbounded}
  Let $\M=\{\mu_1,\ldots,\mu_{\gamma}\}$ be a stochastic multi-mode system with
  a finite set of compactly supported distributions.
  Also, let $S\subset \R^n$ be a path-connected open set. Then, if for any
  non-zero $v\in \R^n$, there exists a mode $\mu\in \M$ such that
  $\bar{\mu}\cdot v>0$, the $\epsilon$-reachability property with
  safety set $S$ holds almost surely.  
\end{corollary}



%% file: algos.tex
\begin{algorithm}[t]\normalsize
  \DontPrintSemicolon
  \KwIn{An $n$-dimensional \smms{} $\M = \{\mu_1, \mu_2, \ldots, \mu_\gamma\}$,
    starting point $x_s$, target point $x_t$, open and path-connected safety set $S$, and
    precision $\varepsilon > 0$} 
  \KwOut{Dynamic reachability algorithm to reach $\varepsilon$ neighborhood of
    $x_t$ using \smms{} $\M$.}

  \If {\textsc{Is\_Almost-Sure\_Reachable($\M$)} = \textsc{No}} {
    \Return  Can not guarantee almost-sure reachability\;
  }
  \Else{
  Compute continuous path $\nu: [0, 1] \to S$ from $x_s$ to $x_t$ using RRT or
  Canny's algorithm\;
  Let $B_1, B_2, \ldots, B_q$ be a finite set of open balls that covers the path
  $\nu$ and stays inside the safety set $S$, $B_i\cap B_{i+1}\not=\emptyset$, $x_s\in B_1$, and $x_t\in B_q$\;
  Let $x_0=x_s,x_1, \ldots, x_{q-1},x_q= x_t$ be set of points such that $x_i \in B_i \cap
  B_{i+1}$ for all $1 \leq i < q$  and $x_s \in B_1$ and $x_t \in B_q$\;
  Set $y(0) = x_s$\;
  Set $k = 0$\;
  \While{$\|y(k) - x_t\| > \varepsilon$}{
    $k := k + 1$\;
    $y(k) :=  \textsc{Reach\_In\_A\_Ball($\M, \varepsilon, y(k), x_k, B_k$)}$\;
  }
  \Return $y(k)$\;
  }
  \caption{\textsc{Reach\_In\_Arb\_PC\_Set}($\M, \varepsilon, x_s, x_t,S$)}
  \label{AlmostSure}
\end{algorithm}

Given a stochastic multi-mode systems $\M$, an arbitrary high-dimensional open-connected safety set
$S$, starting point $x_s \in S$, and target point $x_t \in S$, a typical
hierarchical motion planning procedure for stochastic multi-mode systems include
the following steps:
\begin{enumerate}       
\item
  (path-finding) find a path from $x_s$ to $x_t$,
\item
  (error-margin estimation) find a finite open cover for the path connecting
  $x_s$ to $x_t$ in $S$, and 
\item
  (path-following) compute the  control policy to steer the system from $x_s$ to
  an arbitrary neighborhood of $x_t$ while ensuring safety.
\end{enumerate}
Algorithm~\ref{AlmostSure} provides pseudocode for the this motion planning
problem that invokes Algorithms~\ref{AlmostSure2},~\ref{AlmostSure1},
and~\ref{algo:AlmostSure},  and a call to an off-the-shelf path-finding algorithm. 
There are well-established algorithms to explore non-convex,
high-di\-men\-sional spaces including the rapidly exploring random tree (RRT)
algorithm~\cite{rrt-plan}.
Intuitively, the RRT algorithm can return a path from the source to the destination 
by random exploration of the state space.  This path can be robustly followed by
repeated applications of our algorithm in the context of systems modeled as
\smms{}s by exploiting the fact that $S$ is an open set and the image
$\mathcal{C}$ of the path $\nu$ is compact, and hence find: 
\begin{align}\label{eqn:rstar}
r^*=\min_{x\in \mathcal{C},y\in S^c}\|x-y\|>0,
\end{align}
which exists due to the compactness of $\mathcal{C}$ and closedness of
$S^c=\R^n\setminus S$.
Also, note that if the discovered path is a piece-wise linear path and $S$ is
defined by a set of linear inequalities, then $r^*$ in \eqref{eqn:rstar} can be
lower-bounded by the smallest of the minimum distances of vertices of
$\mathcal{C}$ from the faces of the hyperplanes that define $S$.

\begin{algorithm}[t]\normalsize
  \DontPrintSemicolon
  \KwIn{An $n$-dimensional \smms{} $\M = \{\mu_1, \mu_2, \ldots, \mu_\gamma\}$,
    starting point $x_s$, target point $x_t = 0$, safety Ball $B(x_t, r)$ around
    $x_t$, and precision $\varepsilon > 0$}
  \KwOut{Dynamic reachability algorithm to reach $\varepsilon$ neighborhood of
    the center $x_t=0$ of a given Ball $B(x_t, r)$ using \smms{} $\M$.}

  Set $y(0) := x_s$\;
  Set $k = 0$\;
  Set $\delta = \min \left(\frac{1}{L}, \frac{1}{2Lr},
  \frac{\varepsilon\lambda}{r^2L^2},\frac{\epsilon\lambda}{4rL+2\epsilon\lambda}\right)$ where
  $L = L_{\M}$ is the maximum support of the modes in $\M$ and
  \[\lambda:=\inf_{v\in \R^n: \|v\|=1}\max_{i\in[\gamma]}v\cdot
  \overline{\mu}_i>0.\]
  
  \While{$\|y(k) - x_t\| > \varepsilon$}{
    $k := k + 1$\;
    $m(k) \in \argmax_{i\in[m]}\overline{\mu}_i\cdot (-y(k))$\; 
    $d(k) = {\delta}(r^2-\|y(k)\|^2)$\;
    $r(k) := \textsc{SenseCurrentRate}(y(k-1), m(k), d(k))$\;
    $y(k) := y(k-1) + d(k) \cdot r(k)$\;
  }
  \Return $y(k)$\;
  \caption{\textsc{Reach\_Center($\M, \varepsilon, x_s, x_t, r$)}}
  \label{AlmostSure2}
\end{algorithm}
\begin{algorithm}[t]\normalsize
  \DontPrintSemicolon
  \KwIn{An $n$-dimensional \smms{} $\M = \{\mu_1, \mu_2, \ldots, \mu_\gamma\}$,
    starting point $x_s$, target point $x_t$,  safety set $B(x_0, r)$, and
    precision $\varepsilon > 0$} 
  \KwOut{Dynamic reachability algorithm to reach $\varepsilon$ neighborhood of
    $x_t$ using \smms{} $\M$.}
  Set  $\tilde{r}=\frac{1}{2}\min(r-\|x_s-x_0\|,r-\|x_t-x_0\|)$\;
  Set $N=\lfloor\frac{\|x_t-x_s\|}{\tilde{r}}\rfloor+1$\;
  Set $y(0) = x_s$\;
  Set $k = 0$\;
  \While{$\|y(k) - x_t\| > \varepsilon$}{
    $k := k + 1$\;
    $u(k) := x_s+\frac{k}{N}(x_t-x_s)$\;
    $y(k) := \textsc{Reach\_Center}(\M,
    \min(\varepsilon, \tilde{r}),
    y(k-1), u(k),
    2\tilde{r}
    )$\;
  }
  \Return $y(k)$\;
  \caption{\textsc{Reach\_In\_A\_Ball($\M, \varepsilon, x_s, x_t,
      B(x_0, r)$)}}
  \label{AlmostSure1}
\end{algorithm}

Once $r^*$ is found, let $x_1,\ldots,x_k\in S$ be
such that $x_1=x_s$ and $x_k=x_t$ and
\[
\|x_{i}-x_{i+1}\|\leq r^*/2 \text{ for $i=1,\ldots,k-1$}.
\]
Then, $\{B(x_i,3r^{*}/4)\}_{i\in [k]}$ would be a cover satisfying
\[
B(x_i,3r^{*}/4)\cap B(x_{i+1},3r^{*}/4)\not=\emptyset
\]
and hence,  the proof technique of Theorem~\ref{thrm:mdcase} can be applied to
establish a safe routing from $x_s$ to $x_t$.
The steps 5 and 6 of Algorithm~\ref{AlmostSure} assume the existence of such
sequence of balls and repeatedly invoke Algorithm~\ref{AlmostSure1} to
accomplish reachability within a ball given the stochastic multi-mode system
satisfies $\varepsilon$-reachability property (Algorithm~\ref{algo:AlmostSure}).

Now we turn our focus to analyze computational complexity of deciding
$\varepsilon$-reachability property for path-connected and bounded open safety sets.
For the sake of algorithmic analysis of the problem we assume that, for a given
$\smms{}$, the distributions in all of the modes are computationally tractable,
i.e., the expected vector for each stochastic mode is rational and can be computed
in polynomial time.  
The following complexity result follows from the necessary and sufficient
condition developed in the previous sections. 

\begin{algorithm}[t]\normalsize
  \DontPrintSemicolon
  \KwIn{An $n$-dimensional \smms{} $\M = \{\mu_1, \mu_2, \ldots, \mu_\gamma\}$.}
  \KwOut{\textsc{Yes}, if the $\varepsilon$-reachability property
    holds for path-connected and bounded safety sets, and \textsc{No} otherwise.} 
  Compute the expected multimode system
  \[
  \overline{\M} = \{\overline{\mu_1},
  \overline{\mu_2}, \ldots, \overline{\mu_\gamma}\}\]
  
  Compute 
  \begin{align*}
  P:=\text{min}&\sum_{i=1}^\gamma \|\alpha_i\overline{\mu_i}\|_1 \cr 
  \text{ subject to: } &\alpha_i\geq 1\text{ for all $1\leq i \leq \gamma$}.
  \end{align*}

  \If{($\text{Rank}[\overline{\mu_1},
  	\overline{\mu_2}, \ldots, \overline{\mu_\gamma}]=n)$ and  $(P=0)$}{\Return \textsc{Yes}}
  
  \lElse{\Return \textsc{No}}
	  \caption{\label{algo:AlmostSure}\textsc{Is\_Almost-Sure\_Reachable($\M$)}}
\end{algorithm}

\begin{lemma}
  The decision version of the $\varepsilon$-reachability problem for
  stochastic multi-mode systems is in PTIME.
\end{lemma}
\begin{proof}
  Let $\M = \set{\mu_1, \mu_2, \ldots, \mu_\gamma}$ be an \smms{}
  with a finite set of compactly-supported modes.
  According to Theorem~\ref{thrm:mdcase}, deciding
  $\varepsilon$-reachability is equivalent to deciding whether for all
  $x \in \Rn$ we have that $x\cdot \overline{\mu}_i>0$ for some $i$.

  This later fact is equivalent to showing that any vector
  $x\in\mathbb{R}^n$ can be written as a non-negative linear
  combination of vectors in
  $\{\overline{\mu_1}, \overline{\mu_2}, \ldots,
  \overline{\mu_\gamma}\}$.
  It is well known~\cite{positivecomb} that this property holds if and
  only if the following conditions hold:
  \begin{enumerate}[a.]
  \item The set $\{\overline{\mu_1},
    \overline{\mu_2}, \ldots, \overline{\mu_\gamma}\}$ spans $\mathbb{R}^n$, and 
  \item $\sum_{i=1}^\gamma\alpha_i\overline{\mu_i}=0$ for some real numbers $\alpha_i\geq 1$.
  \end{enumerate} 
  This observation implies the correctness of
  Algorithm~\ref{algo:AlmostSure} in deciding the almost-sure
  $\varepsilon$-reachability problem for stochastic multi-mode
  systems.
 
  Note that the key computation effort in the algorithm is to check
  the solution of a linear program and to verify the full-rank
  condition of a matrix.  Since both linear
  programming~\cite{megiddo1986complexity} and matrix rank
  computation~\cite{ibarra1982generalization} can be solved in
  polynomial time, it follows that almost-sure
  $\varepsilon$-reachability problem can be decided in polynomial
  time.
\end{proof}

\begin{figure}[t]
  \centering
  	\subfloat[]{\includegraphics[width=0.49\textwidth]{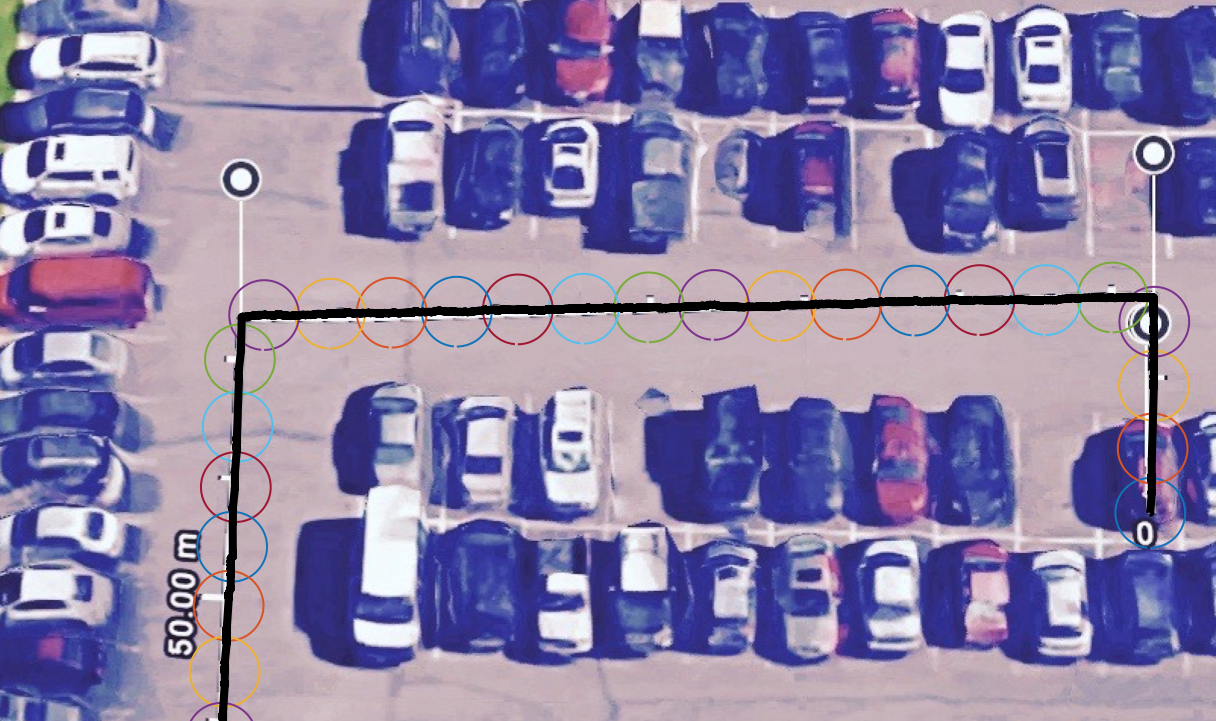}\label{fig:sim-resultsa}}\\
    \subfloat[]{\includegraphics[width=0.49\textwidth,clip,trim=1cm 0 0 0.1cm]{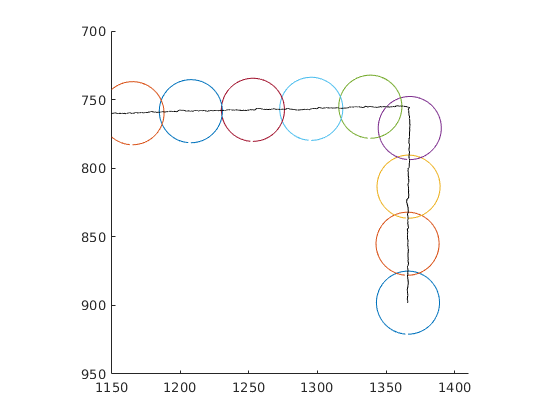}\label{fig:sim-resultsb}}
  \caption{Simulation result of a self-driving car with stochastic modes.}
  \label{fig:sim-results}
\end{figure}

Figure~\ref{fig:sim-results} shows a simulation of a two-dimensional
\smms{} that models a self driving car with stochastic control
directions. This model consists of modes
$\M=\{(0,1)+U,(1,0)+U,(0,-1)+U,(-1,0)+U\}$, where $U$ is a uniform
distribution over the $[-7.5,7.5]\times[-7.5,7.5]$ box in $\R^2$. Note that the
variance of the noise is much higher than the length of the expected
control directions, i.e., the plant has a very low signal to noise
ratio. The control objective is to take the (red) car out of the
parking lot while maintaining safety. Figure~\ref{fig:sim-resultsa}
shows a sample path generated by the proposed algorithm. The circles
mark the width of the safe region. As it can be seen in that picture,
safety is maintained while taking the car out of the parking
lot. Figure~\ref{fig:sim-resultsb} shows a zoomed-in view of the
same sample path.


%% file: conclusion.tex
We introduced and studied stochastic multi-mode systems, which are 
a natural formalism in motion planning when hierarchical control is
combined with noisy sensors or actuators.
A key result of this paper is that it is possible to efficiently decide the
$\varepsilon$-reachability problem, that is, whether a control strategy exists 
that steers a stochastic multi-mode system from any point in an open,
path-connected set to an arbitrary neighborhood of any other point in that set with probability one.
We have shown in particular that a stochastic system enjoys the
$\varepsilon$-reachability property if and only if the associated deterministic
system does.
This condition implies that for almost-sure $\varepsilon$-reachability of
stochastic multi-mode systems can be checked efficiently. 

As a natural next step we are investigating the possibility of employing
decision procedures for stochastic multi-mode systems in a counterexample-guided
abstraction-refinement framework in order to develop motion planning algorithms
for systems with richer dynamics (linear-hybrid systems) in the presence of
stochastic uncertainties.


%% file: main.bbl
\begin{thebibliography}{10}

\bibitem{AH07}
A.~Aguiar and J.~Hespanha.
\newblock Trajectory-tracking and path-following of underactuated autonomous
  vehicles with parametric modeling uncertainty.
\newblock {\em Automatic Control}, 52(8):1362--1379, 2007.

\bibitem{Alur95thealgorithmic}
R.~Alur, C.~Courcoubetis, N.~Halbwachs, T.~A. Henzinger, P.-H. Ho, X.~Nicollin,
  A.~Olivero, J.~Sifakis, and S.~Yovine.
\newblock The algorithmic analysis of hybrid systems.
\newblock {\em Theoretical Computer Science}, 138:3--34, 1995.

\bibitem{ACHH92}
R.~Alur, C.~Courcoubetis, T.~A. Henzinger, and P.-S. Ho.
\newblock Hybrid automata: An algorithmic approach to the specification and
  verification of hybrid systems.
\newblock In {\em Hybrid Systems}, pages 209--229, 1992.

\bibitem{alurDill94}
R.~Alur and D.~Dill.
\newblock A theory of timed automata.
\newblock {\em Theoretical Computer Science}, 126:183--235, 1994.

\bibitem{AFMT13}
R.~Alur, V.~Forejt, S.~Moarref, and A.~Trivedi.
\newblock Safe schedulability of bounded-rate multi-mode systems.
\newblock In {\em HSCC}, pages 243--252, 2013.

\bibitem{ATW12}
R.~Alur, A.~Trivedi, and D.~Wojtczak.
\newblock Optimal scheduling for constant-rate multi-mode systems.
\newblock In {\em HSCC}, pages 75--84, 2012.

\bibitem{AMP95}
E.~Asarin, M.~Oded, and A.~Pnueli.
\newblock Reachability analysis of dynamical systems having piecewise-constant
  derivatives.
\newblock {\em TCS}, 138:35--66, 1995.

\bibitem{belta2007symbolic}
C.~Belta, A.~Bicchi, M.~Egerstedt, E.~Frazzoli, E.~Klavins, and G.~J. Pappas.
\newblock Symbolic planning and control of robot motion [grand challenges of
  robotics].
\newblock {\em IEEE Robotics \& Automation Magazine}, 14(1):61--70, 2007.

\bibitem{BCKO08}
M.~d. Berg, O.~Cheong, M.~v. Kreveld, and M.~Overmars.
\newblock {\em Computational Geometry: Algorithms and Applications}.
\newblock Springer-Verlag TELOS, Santa Clara, CA, USA, 3rd ed. edition, 2008.

\bibitem{BJKST15}
D.~Bhave, S.~Jha, S.~N. Krishna, S.~Schewe, and A.~Trivedi.
\newblock Bounded-rate multi-mode systems based motion planning.
\newblock In {\em Proceedings of the 18th International Conference on Hybrid
  Systems: Computation and Control}, HSCC '15, pages 41--50, New York, NY, USA,
  2015. ACM.

\bibitem{BBM98}
M.~S. Branicky, V.~S. Borkar, and S.~K. Mitter.
\newblock A unified framework for hybrid control: Model and optimal control
  theory.
\newblock {\em Automatic Control}, 43(1):31--45, 1998.

\bibitem{Can88}
J.~F. Canny.
\newblock {\em The Complexity of Robot Motion Planning}.
\newblock MIT Press, Cambridge, MA, USA, 1988.

\bibitem{durrett2010probability}
R.~Durrett.
\newblock {\em Probability: theory and examples}.
\newblock Cambridge University Press, 2010.

\bibitem{Firby89}
R.~J. Firby.
\newblock {\em Adaptive Execution in Complex Dynamic Worlds}.
\newblock PhD thesis, Yale University, New Haven, CT, USA, 1989.
\newblock AAI9010653.

\bibitem{folland2013real}
G.~B. Folland.
\newblock {\em Real analysis: modern techniques and their applications}.
\newblock John Wiley \& Sons, 2013.

\bibitem{FDF00}
E.~Frazzoli, M.~A. Dahleh, and E.~Feron.
\newblock Robust hybrid control for autonomous vehicle motion planning.
\newblock In {\em Decision and Control, 2000. Proceedings of the 39th IEEE
  Conference on}, volume~1, pages 821--826. IEEE, 2000.

\bibitem{FDF05}
E.~Frazzoli, M.~A. Dahleh, and E.~Feron.
\newblock Maneuver-based motion planning for nonlinear systems with symmetries.
\newblock {\em IEEE Transactions on Robotics}, 21(6):1077--1091, Dec 2005.

\bibitem{Gat98}
E.~Gat.
\newblock Three-layer architectures.
\newblock In D.~Kortenkamp, R.~P. Bonasso, and R.~Murphy, editors, {\em
  Artificial Intelligence and Mobile Robots}, pages 195--210. MIT Press,
  Cambridge, MA, USA, 1998.

\bibitem{HKPV98}
T.~A. Henzinger, P.~W. Kopke, A.~Puri, and P.~Varaiya.
\newblock What's decidable about hybrid automata?
\newblock {\em Journal of Comp. and Sys. Sciences}, 57:94--124, 1998.

\bibitem{ibarra1982generalization}
O.~H. Ibarra, S.~Moran, and R.~Hui.
\newblock A generalization of the fast lup matrix decomposition algorithm and
  applications.
\newblock {\em Journal of Algorithms}, 3(1):45--56, 1982.

\bibitem{latombe2012robot}
J.-C. Latombe.
\newblock {\em Robot motion planning}, volume 124.
\newblock Springer Science \& Business Media, 2012.

\bibitem{Lav06}
S.~M. LaValle.
\newblock {\em Planning Algorithms}.
\newblock Cambridge University Press, Cambridge, U.K., 2006.
\newblock Available at http://planning.cs.uiuc.edu/.

\bibitem{rrt-plan}
S.~M. LaValle and J.~J. Kuffner.
\newblock {Randomized kinodynamic planning}.
\newblock In {\em Robotics and Automation}, volume~1, pages 473--479. IEEE,
  1999.

\bibitem{LeNP12}
J.~Le~Ny and G.~Pappas.
\newblock Sequential composition of robust controller specifications.
\newblock In {\em Robotics and Automation}, pages 5190--5195, 2012.

\bibitem{megiddo1986complexity}
N.~Megiddo.
\newblock {\em On the complexity of linear programming}.
\newblock IBM Thomas J. Watson Research Division, 1986.

\bibitem{Poznyak2009}
A.~S. Poznyak.
\newblock {Martingales}.
\newblock In {\em Advanced Mathematical Tools for Automatic Control Engineers},
  pages 133--173. Elsevier Inc., 2009.

\bibitem{positivecomb}
R.~G. Regis.
\newblock On the properties of positive spanning sets and positive bases.
\newblock {\em Optimization and Engineering}, 17(1):229--262, 2016.

\end{thebibliography}
